\documentclass[a4paper, reqno]{amsart}
\usepackage[utf8]{inputenc}
\usepackage[T1]{fontenc}
\usepackage{amssymb}
\usepackage{amsfonts}
\usepackage{amsmath, amsthm, mathabx, graphicx}
\usepackage{tensor}
\usepackage[new]{old-arrows}
\usepackage{mathrsfs}
\usepackage{hyperref}
\usepackage{tikz}
\usetikzlibrary{matrix, arrows.meta}
\usetikzlibrary{cd}
\usepackage{pdfpages}

\newtheorem{sz}{Satz}[section]

\newtheorem{thr}[sz]{Theorem}
\newtheorem{dt}[sz]{Definition}

\newtheorem{pr}[sz]{Proposition}
\newtheorem{re}[sz]{Remark}
\newtheorem{co}[sz]{Corollary}

\theoremstyle{definition}

\newtheorem*{prob}{Problem}
 
\def\textmap#1{\mathop{\vbox{\ialign{
                                  ##\crcr
      ${\scriptstyle\hfil\;\;#1\;\;\hfil}$\crcr
      \noalign{\kern 1pt\nointerlineskip}
      \rightarrowfill\crcr}}\;}}

\def\qmod#1#2{{\hbox{}^{\displaystyle{#1}}}\!\big/\!\hbox{}_{
\displaystyle{#2}}}

\def\C{{\mathbb C}}

\def\H{{\mathbb H}}

\def\R{{\mathbb R}}
\def\Z{{\mathbb Z}}
\def\ag{{\mathfrak a}}
\def\bg{{\mathfrak b}}

\def\g{{\mathfrak g}}
\def\hg{{\mathfrak h}}
\def\ig{{\mathfrak i}}

\def\mg{{\mathfrak m}}

\def\tg{{\mathfrak t}}

\def\id{\mathrm{id}}
\def\End{\mathrm{End}}
\def\Hom{\mathrm{Hom}}
\def\Aut{\mathrm{Aut}}
\def\SL{\mathrm{SL}}
\def\SU{\mathrm{SU}}
\def\PSL{\mathrm{PSL}}
\def\GL{\mathrm{GL}}
\def\Spin{\mathrm{Spin}}

\def\SO{\mathrm{SO}}
\def\PU{\mathrm{PU}}
\def\so{\mathfrak{so}}
\def\su{\mathrm{su}}

\def\g{\mathfrak{g}}
\def\Ad{\mathrm{Ad}}
\def\ad{\mathrm{ad}}
\newcommand{\extpw}{\mathchoice{{\textstyle\bigwedge}}%
    {{\bigwedge}}%
    {{\textstyle\wedge}}%
    {{\scriptstyle\wedge}}}
\def\extp{{\extpw}\hspace{-2pt}}

\newcommand\cal{\mathcal}

\def\edf{\coloneq}

\begin{document}

\title[Reductive homogeneous spaces associated with real forms] {Reductive homogeneous spaces associated with real forms. A gauge-theoretical generalisation} 
\author{Nicolas Al Choueiry}
\address{Aix Marseille Université, CNRS, Centrale Marseille, I2M, UMR 7373, 13453 Marseille, France}
\email[Nicolas Al Choueiry]{nicolas.choueiry@hotmail.com}
\author{Andrei  Teleman}
\email[Andrei  Teleman]{andrei.teleman@univ-amu.fr}

\begin{abstract}

Let $G$ be a connected complex Lie group. A real form of $G$ is a closed subgroup $H\subset G$ whose Lie algebra $\hg$ is a real form of the Lie algebra $\g$ of $G$. A pair $(G,H)$ of this type is reductive, and the corresponding quotient $G/H$ is a reductive homogeneous space whose canonical connection is torsion free. 

Regarded as a  principal $H$-bundle over $G/H$, $G$ comes with  tensorial 1-form $\alpha$ of type $\Ad$ and a natural left invariant connection $A$. 
This remark suggests the following natural gauge theoretical generalisation of the class of reductive pairs of the form $(G,H)$ as above: 

Let $H$ be an arbitrary Lie group.  A triple $(P\textmap{\pi}M,\alpha,A)$, where $P\textmap{\pi}M$ is a principal  $H$-bundle,  $\alpha$   a tensorial 1-form  of type $\Ad$ on $P$ and $A$ a connection on $P$ will be called admissible if the induced linear maps $A_y\to \hg$, $y\in P$, are all isomorphisms. If this is the case one obtains a canonical linear connection $\nabla^\alpha_A$ on $M$ and a canonical   almost complex structure $J^\alpha_A$ on $P$ which,   by a result of R. Zentner, is integrable if an only if the pair $(\alpha,A)$ satisfies a  gauge invariant system of two first order differential equations. 
A triple $(P\textmap{\pi}M,\alpha,A)$ as above will be called integrable, or Zentner triple, if this integrability condition for $J^\alpha_A$ is satisfied.

The main result of \cite{Chou} states that any integrable triple $(P\textmap{\pi}M,\alpha,A)$ with $M$ simply connected  and $\nabla^\alpha_A$ complete can be identified with the triple associated with a real form of a complex Lie group;  in particular there exists a complex Lie group $G$, a real form $H'\subset G$ of $G$ with $H'\simeq H$ such that the pair $(M,\nabla^\alpha_A)$ can be identified with the reductive homogeneous space $G/H'$ endowed with its canonical connection.  In this article we explain the strategy of the proof  of this classification result and we prove in detail a theorem which plays an important role in this strategy and is of independent interest.  In the last section we introduce the moduli spaces of integrable pairs on a principal bundle, and we give explicit examples. 
 
  \end{abstract}
\maketitle
\begin{quotation}
	
\end{quotation}

\tableofcontents

\section{Introduction}\label{intro}

A pair $(G,H)$ consisting of a Lie group $G$ and a closed subgroup $H$ is called reductive if the Lie subalgebra $\hg\subset \g$ of $H$ has a complement $\mg$ which is invariant under the adjoint representation $\Ad_H:H\to \GL(\g)$ of $H$.

Let  $(G,H)$ be a reductive pair and $\mg$ a fixed $\Ad_H$-invariant complement of $\hg$. We will denote by $\Ad_H^\mg:H\to \GL(\mg)$ the representation of $H$ on $\mg$ which is induced by  the adjoint $\Ad_H$ of $H$ on $\g$.

The canonical $\g$-valued left invariant form $\omega\in A^1(G,\g)$ \cite[p. 41]{KN1} on $G$ decomposes as $\omega=\theta+\ag$, where the forms $\theta\in A^1(G,\hg)$, $\ag\in   A^1(G,\mg)$ are also left invariant; these forms play an important role in the theory of reductive homogeneous spaces.

The form $\ag\in   A^1_{\Ad_H^\mg}(G,\mg)$ is an $\mg$-valued tensorial tensorial form of type $\Ad_H^\mg$ on $G$ viewed as principal bundle over  $G/H$ (see \cite[Example 5.1 p. 55]{KN1}), so it can be regarded as 1-form on  $G/H$ with values in the associated vector bundle $E\edf G\times_{\Ad_H^\mg}\mg$ on $G/H$. For any point $x\in G/H$ the linear map $\ag_x:T_x(G/H)\to E_x$ is an isomorphism, so $\ag$ defines an isomorphism $\bg:T_{G/H}\to E$ of  vector bundles  on $G/H$.

The form $\theta \in A^1(G,\hg)$ is the connection form of a   left invariant connection $A\subset T_G$ on the principal bundle $\pi:G\to G/H$ whose fibre $A_g$ at a point $g\in G$ is the left translation $l_{g*}(\mg)$ of $\mg$.

The connection $A$ induces a linear connection $\nabla_A^{\Ad_H^\mg}$ on the associated bundle $E$, so it also induces a linear connection $\nabla:=\bg^{-1}(\nabla_A^{\Ad_H^\mg})$ on the tangent bundle $T_{G/H}$ of the homogenous space $G/H$.   It coincides with Nomizu's canonical connection of the second kind on reductive homogeneous space $G/H$  \cite[p. 49]{No}, which  can be  characterised as follows:
\begin{itemize}
\item $\nabla$ is the unique  left invariant connection on $G/H$ satisfying  condition \cite[(A2) p. 47]{No}. In particular $\nabla$ satisfies condition \cite[(A1) p. 47]{No} and is geodesically complete.
\item via Nomizu's classification theorem of the left invariant connections on a reductive homogeneous space \cite[Theorem 8.1]{No}, $\nabla$ is associated with the trivial $H$-invariant bilinear form $\mg\times\mg\to\mg$ \cite[Theorem 10.2]{No}. In particular the torsion and the curvature of $\nabla$ are given by the formulae
\begin{equation}\label{torsion-curvature}
T^\nabla(X,Y)=-[X,Y]_{\mg}, \ R(X,Y)(Z)=-[[X,Y]_\hg,Z],
\end{equation}
for $X$, $Y$, $Z\in\mg$, where the subscripts $_{\mg}$, $_{\hg}$ stand for the projections on the respective subspaces \cite[section 9]{No}.
\end{itemize}

Note that Nomizu also introduces the canonical connection of the first kind associated with $\mg$: it is the unique left invariant {\it torsion free} connection $\nabla'$  on $G/H$ satisfying  condition \cite[(A1) p. 47]{No}.
\vspace{2mm}

This article will focus on a special class of reductive pairs: the pairs formed by a complex Lie group $G$ and a real form of $G$ in the following sense:
\begin{dt}
Let $G$ be a complex Lie group. A closed subgroup $H\subset G$ will be called a real form of $G$ if its Lie algebra $\hg$ is a real form of $\g$, i.e. if one of the following equivalent conditions is satisfied:
\begin{enumerate}
	\item $\hg$ and $i\hg$ give a direct sum decomposition of $\g$.
	\item the canonical $\C$-linear map $\hg\otimes_\R\C\to \g$ induced by the inclusion morphism $\hg\hookrightarrow \C$ is an isomorphism. 
\end{enumerate} 
\end{dt}

Note that
\begin{re}
Let $H$ be a real form of $G$. Then $i\hg$ is an $\Ad_H$-invariant complement of $\hg$ in 	$\g$, in particular $(G,H)$ is a reductive pair.
\end{re}
\begin{proof}
Indeed, let $h\in H$. The inner automorphism $\iota_h\in \Aut(G)$ is holomorphic, so its differential $\Ad_h=(\iota_h)_{*e}$ at the unit element $e\in G$ is $\C$-linear. Since $\Ad_h$ leaves $\hg$ invariant, it will also leave invariant $i\hg$.	
\end{proof}

Therefore
\begin{re}\label{alpha-A-rem}
For a real form $H$ of $G$ we have a canonical choice of a	 complement of $\hg$ in $\g$, so we obtain
\begin{enumerate} 
\item A canonical tensorial 1-form $\ag\in A^1_{\Ad_H^\mg}(G,i\hg)$ of type $\Ad_H^{i\hg}$. Using the obvious isomorphism $\hg\to i\hg$, we can write $\ag=-i\alpha$ for tensorial form  $\alpha\in A^1_{\Ad_H}(G,\hg)$ of type $\Ad_H$.
\item A  canonical left invariant connection $A$ on the principal $H$-bundle  $\pi:G\to  G/H$.	

\end{enumerate}
\end{re}

Therefore one also obtains a canonical geodesically complete connection $\nabla$ on the quotient $G/H$ which is induced by $A$ via the vector bundle isomorphism $\bg$ associated with $\ag$.

Note also that $[i\hg,i\hg]\subset  \hg$, so formula (\ref{torsion-curvature}) gives 
\begin{equation}\label{torsion-curvature-new}
T^\nabla(X,Y)=0, \ R(X,Y)(Z)=-[[X,Y],Z],
\end{equation}
for any $X$, $Y$, $Z\in i\hg$. In particular, taking into account, \cite[Theorem 10.1]{No} we see that $\nabla=\nabla'$. In other words
\begin{re}
Let $H$ be a real form of $G$.   Nomizu's   canonical connections $\nabla$, $\nabla'$ on $G/H$ associated with the  canonical complement $i\hg$	of $\hg$ coincide. 
\end{re}

The connection $\nabla=\nabla'$ will be called the canonical connection on the homogeneous space $G/H$.

The (real) tangent bundle $T_G$ of $G$ comes with a left invariant direct sum decomposition $T_G=A\oplus V$, where $V$ is the vertical tangent sub-bundle. The almost complex structure $J_G$ which defines the complex structure of $G$ intertwines these two sub-bundles. More precisely let $j_e^H:A_e=i\hg \to \hg=V_e$ be the isomorphism defined by the multiplication by $i$ and let $j^H:A\to V$ be the left invariant extension  extension of $j_e^H$.  Then, with respect to the  decomposition $T_G=A\oplus V$, we have
$$
J_G=  \begin{pmatrix}
    0 & -(j^{H})^{-1} \\
    j^{H} & 0 
    \end{pmatrix}. 
    $$ 
Now note that the bundle isomorphism $j^H:A\to V$ can be expressed explicitly in terms of the pair $(\alpha,A)$ which interviens in Remark \ref{alpha-A-rem}. More precisely, the restriction of $\alpha$ to $A$ defines an isomorphism $A_g { \textmap{\simeq a_g}}\hg$ for any  $g\in G$.  On the other hand for $g\in G$ we also have the canonical isomorphism $\hg\textmap{\simeq \#_g}V_g$ induced by the right action of $H$ on $G$ and which is given explicitly by $\zeta\mapsto \zeta^\#_g$.  
\begin{re}
The bundle isomorphism $j^H:A\to V$ is given by 	family of isomorphisms $(\#_g\circ a_g)_{g\in G}$.
\end{re}

The starting point of this article is the following remark which plays a crucial role in \cite{Chou}: for any Lie group $H$, any principal $H$-bundle $\pi:P\to M$ and any pair $(\alpha,A)$ consisting of a tensorial 1-form $\alpha\in A^1_{\Ad_H}(P,\hg)$ \cite[section II.5]{KN1} of type $\Ad$ satisfying a natural non-degeneracy condition, one can define a bundle isomorphism $j^\alpha_A:A\to V$ in a similar way, so one obtains a canonically associated almost complex structure $J^\alpha_A$ on $P$. This construction will be detailed in the next section.

 \section{Integrable triples and associated complex structures on principal bundles}
 
 \subsection{Admissible triples. Integrable triples} 
 
 Let $H$ be a Lie group and  $P\textmap{\pi} M$ be a principal $H$-bundle.
\begin{dt}\label{admiss-def}
A tensorial 1-form $\alpha$ of type $\Ad$ on $P$ will be called admissible if for any $y\in P$, the linear map ${T_yP}/{V_y}   \textmap{a_y} \hg$ induced by $\alpha_y$ is an isomorphism. 
\end{dt}

Note that (see \cite[Remark 1.2.1]{Chou}):

\begin{re}\label{alpha-beta}
The data of an admissible 1-form   of type $\Ad$ on $P$ is equivalent to the data of a vector bundle isomorphism $T_M\textmap{\simeq} \Ad(P)\coloneq P\times_\Ad\hg$.	
\end{re}
The isomorphism $\beta: T_M\to \Ad(P)$ associated with an admissible form $\alpha$ is defined as follows: for any $x\in M$, $v\in T_xM$, we have $\beta(v)=[y,\alpha(\tilde v)]$, where $y\in P_x$ and $\tilde v\in T_yP$ is a lift of $v$ at $y$. 

Using the bundle isomorphism $\beta$  we can endow the tangent bundle $T_M$ with the linear connection $ \nabla^\alpha_A=\beta^{-1}(\nabla^\Ad_A)$, the pull back via $\beta$ of the linear connection $\nabla^\Ad_A$ on $\Ad(P)$ which is associated with $A$ in the sense of \cite[section III.1]{KN1}, \cite[section 3.2]{Te}.

Note that the canonical structure
$$[\cdot,\cdot]:\Ad(P)\times_M \Ad(P)\to\Ad(P)$$
of a Lie algebra bundle on $\Ad(P)$ is $\nabla^\Ad_A$-parallel, so:

\begin{re}\label{[]est-nabla-parall}
The Lie algebra bundle structure $[\cdot,\cdot]_\alpha: T_M \times_M T_M \to T_M$ on $T_M$ induced by $[\cdot,\cdot]$ via $\beta$ is $\nabla_A^\alpha$-parallel.
\end{re}

Let $\alpha \in A^1_{\Ad}(P, \mathfrak{h})$ be an admissible  tensorial 1-form of type $\Ad$, and $A$ a connection on $P$. Composing the fibrewise isomorphic map  $ T_P/V \to   \mathfrak{h}$ induced by $\alpha$ with the obvious isomorphism
$$\begin{tikzcd}
A \ar[r, hook] \ar[rr, bend right = 20, "\simeq"'] & T_P \ar[r, two heads] & T_P/V
\end{tikzcd}
$$
we obtain a fibrewise isomorphic map $a:A\to \hg$. On the other hand the right action of $H$ on $P$ defines a family of isomorphisms $(\#_y:\hg\to V_y)_{y\in P}$. The family of compositions $(\#_y\circ a_y)_{y\in P}$ defines a bundle isomorphism $j^\alpha_A:A\to V$.
\begin{dt}\label{JalphaA}
Let $\pi:P\to M$ be a principal $H$-bundle, let $\alpha$ be an admissible  tensorial 1-form of type $\Ad$ on $P$ and $A$ a connection on $P$. The almost complex structure associated with the triple $(P \textmap{\pi} M, \alpha, A)$ is the almost complex structure $J^{\alpha}_A: T_P \to T_P$ on $P$ defined with respect to the decomposition $T_P=A\oplus V$ by :
$$
(J^{\alpha}_A)_y   \edf \begin{pmatrix}
    0 & -(j^{\alpha}_{A,y})^{-1} \\
    j^{\alpha}_{A,y} & 0 
    \end{pmatrix}. 
    $$
\end{dt}

The integrability  of  $J^{\alpha}_A$ has been studied in \cite{Ze}. His result states 
\begin{thr}\cite{Ze}
\label{ZentTh} The almost complex structure $J^{\alpha}_A$ is integrable (so it defines a complex manifold structure on $P$) if and only if $(\alpha,A)$ is a solution of the differential system
\begin{equation} \label{ZentEq}
\left\{\begin{array}{ccc}
d_A \alpha &  = & 0 \\ 
\Omega_A & =&   \frac{1}{2}[ \alpha \wedge \alpha].
\end{array}\right. \tag{$Z$}
\end{equation}
where $\Omega_A \in A^2_{\Ad}(P,\mathfrak{h})$ is the curvature form of $A$.
\end{thr}

In the first equation, the symbol $d_A$ on the left  stands for the de Rham differential %

$$d_A:A^r_\Ad(P,\hg)=A^r(M,\Ad(P)\to A^{r+1}(M,\Ad(P)= A^{r+1}_\Ad(P,\hg)$$
associated with $A$ (\cite[section 3.2]{Te}, \cite[section II.5]{KN1}). In the second equation on the right we used the notation $[\cdot\wedge\cdot]$ for the bilinear map 
$$A^1_\Ad(P,\hg)\times A^1_\Ad(P,\hg)\to A^2_\Ad(P,\hg),\ 
[\alpha,\beta](v,w)\edf [\alpha(u),\beta(v)]-[\alpha(v),\beta(u)].
$$
Definition \ref{JalphaA} and  Theorem \ref{ZentTh} justify the following
\begin{dt}\label{DefTripAdmTripZentner}
Let $H$ be Lie group. 

An admissible $H$-triple  is a triple $(P \textmap{\pi} M,\alpha,A)$, where $\pi: P \to M$ is a principal $H$-bundle, $\alpha \in A^1_{\Ad}(P,\mathfrak{h})$ is an admissible 1-form of   type $\Ad$, and $A$ is a connection on $P$.
An admissible $H$-triple $(\pi: P \to M,\alpha,A)$ will be  called integrable (or Zentner triple) if $(\alpha,A)$ satisfies the integrability conditions (Z).
\end{dt}

Note that any admissible integrable $H$-triple $(P \textmap{\pi} M,\alpha,A)$ gives a a complex manifold $(P,J^\alpha_A)$ of complex dimension $\dim(H)$ which is compact when $H$ is compact. Therefore the classification of integrable triples (with compact structure group) is related to the classification   of (compact) complex manifolds.
\vspace{2mm}

If $(P \textmap{\pi} M,\alpha,A)$ is an integrable admissible triple, then the torsion of the linear connection $\nabla_A^\alpha$ vanishes, and we have an explicit formula for the curvature of $\nabla_A^\alpha$ in terms of the Lie algebra bundle structure $[\cdot,\cdot]_\alpha$ defined above.   Propositions \cite[Propositions 2.1, 2.2]{Ze} require no compatibility condition with a Riemannian metric and can be interpreted as follows:

\begin{pr} \label{courbure-nablaM}
Let $(\pi: P \to M,\alpha,A)$ be an integrable triple. Then the torsion $T^{\nabla_A^\alpha}$ of $\nabla_A^\alpha$ is zero, and the curvature $R^{\nabla_A^\alpha} \in A^2(M,\End(T_M))$ of $\nabla_A^\alpha$ is given by the formula
$$
R^{\nabla_A^\alpha}(u,v)(w) = [[u,v]_\alpha,w]_\alpha.
$$ 	
\end{pr}

\subsection{Examples}

In this section, we will present two classes of exemples of integrable triples: the triples associated with real forms of complex Lie groups, and the triples associated with oriented hyperbolic 3-manifolds.

\subsubsection{Integrable triples associated with real forms}\label{triples-real-forms}

Let $G$ be complex Lie group, $\omega\in A^1(G,\g)$ its canonical left invariant $\g$-valued 1-form, and $H$ a real form of $G$ (see section \ref{intro}). Using the direct sum decomposition $\mathfrak{g} = \mathfrak{h} \oplus i\mathfrak{h}$, we obtain the decomposition
$$
\omega = \theta_{H} - i\alpha_{H},
$$
where $\theta_{H} \edf \Re(\omega) \in A^1(G, \mathfrak{h})$, $\alpha_{H} = -\Im(\omega) = \Re(i\omega)$, and where $\Re$ ($\Im$) denote the real (respectively imaginary) part associated to the real structure on $\mathfrak{g}$ defined by the real form $\mathfrak{h}$.

The arguments explained in section \ref{intro} give (see \cite[section 1.3.1]{Chou} for details): 

\begin{pr}\label{TripletsAss}
Regard the canonical projection  $G\to G/H$ as a principal $H$-bundle. Then:
\begin{enumerate}
	\item $\theta_{H}$ is a connection form on the principal bundle $G \to G/H$.
	\item $\alpha_{H}$ is an admissible  tensorial 1-form of type $\Ad$ on the principal $H$-bundle $G \to G/H$.
	\item Let $A_{H}$ be the connection defined by $\theta_{H}$. The almost complex structure $J^{\alpha_{H}}_{A_{H}}$ associated (in the sense of Definition \ref{JalphaA}) with the admissible triple $(G \to G/H, \alpha_{H}, A_{H})$ coincides with the canonical almost complex structure of $G$.
	\item The triple $(G \to G/H, \alpha_{H}, A_{H})$ is an integrable triple.
\end{enumerate}
\end{pr}

\subsubsection{Integrable triples associated with hyperbolic 3-manifolds}

Recall first that an oriented Euclidean space $(E, \langle\cdot,\cdot\rangle)$  of dimension 3  is naturally endowed with a cross product
$$
\times: E \times E \to E
$$   
defined by the identity
$$
\langle x \times y, z \rangle = \mathrm{vol}(x, y, z),
$$
where $\mathrm{vol} \in \wedge^3(E^*)$ denotes the volume form of $(E, \langle\cdot,\cdot\rangle)$. Note that $(E, \times)$ is a Lie algebra. Let $\mathfrak{so}(E)$ be the Lie algebra of antisymmetric endomorphisms of $(E, \langle\cdot,\cdot\rangle)$. The map
$$
\alpha: E \to \mathfrak{so}(E)
$$
defined by $\alpha(x)(v) = x \times v$ is a Lie algebra isomorphism. In particular, we have the identity
\begin{equation}\label{isoAlgLie}
[\alpha(u), \alpha(v)](w) = (u \times v) \times w.
\end{equation}

Now let $(M, g)$ be an oriented  3-dimensional  Riemannian manifold. The tangent bundle $T_M$ is an oriented    Euclidean vector bundle of rank 3, so it is naturally endowed with the cross product
$$
\times: T_M \times_M T_M \to T_M
$$
which makes $T_M$ a Lie algebra bundle and satisfies the identity
$$
g_x(u \times v, w) = \mathrm{vol}_g(u, v, w),
$$
for any $x \in M$, $(u,v,w) \in T_xM^3$, where $\mathrm{vol}_g \in A^3(M, \mathbb{R})$ stands for the volume form of the oriented Riemannian manifold $(M, g)$. With these preparations we have (see \cite[Proposition 6.1, p. 133]{Gh}, \cite[section 1.3.2]{Chou}):

\begin{pr}\label{courbure-constante} Let $(M,g)$ be an oriented  3-dimensional  Riemannian manifold, $\nabla$ its Levi-Civita connection and $d^\nabla$ the associated de Rham differential.
\begin{enumerate}
\item 	The 1-form $\alpha \in A^1(M, \mathfrak{so}(T_M))$  defined by $\alpha(u)(v) = u \times v$ satisfies the equation $d^\nabla(\alpha) = 0$.
\item Suppose $g$ is a Riemannian metric of constant sectional curvature $\kappa$. Then the curvature form $R \in A^2(M, \mathfrak{so}(T_M))$ of $g$ is given by the formula
\begin{equation}\label{CourbVarHyp}
R(u, v)(w) = -\kappa (u \times v) \times w,
\end{equation}
which can be written more compactly as
\begin{equation}\label{CourbVarHypFormeCompacte}
R = -\frac{\kappa}{2}[\alpha \wedge \alpha].
\end{equation}

\end{enumerate}
\end{pr}

This implies:

\begin{co}\label{hyperbolic-coro}
Let $(M, g)$ be an oriented 3-dimensional Riemannian manifold, let $P_g$ be the $\mathrm{SO}(3)$-principal bundle of orthonormal frames compatible with the orientation, and let $A$ be the connection on $P_g$ corresponding to the Levi-Civita connection on $T_M$.

Assume that $g$ has constant sectional curvature $\kappa$. The triple $(P_g \to M, \alpha, A)$ is integrable if and only if $\kappa = -1$. If this is the case, we will call it the integrable tripe associated with $(M,g)$. 

In particular, the 6-dimensional manifold $P_g$ associated with any oriented Riemannian 3-manifold $(M,g)$ of constant sectional curvature $-1$  carries a canonical complex structure.
\end{co}

The special case of the hyperbolic space is interesting for us, because it yields an integrable triple which is also associated with a real form. More precisely:
 
\begin{re} \label{H3-real-form}
 The integrable triple associated with the hyperbolic space $(\H^3, g_{\H^3})$ corresponds, via a natural identification, to the integrable triple associated with the pair $(\PSL(2,\C), \PU(2))$, where $\PU(2)$ is considered as a real form of the complex Lie group $\PSL(2,\C)$. 
\end{re}

The principal bundle
$$
\PSL(2,\C) \to \PSL(2,\C)/\PU(2) = \H^3
$$
associated with the pair $(\PSL(2,\C), \PU(2))$ is the fundamental object intervening  in the famous Bryant theorem on surfaces of constant mean curvature 1 (CMC1 surfaces) in hyperbolic space \cite{Br}.

Note that   the hyperbolic $\H^3$ is simply connected and complete. It is natural to ask if this is a general phenomenon, i.e. if any integrable triple $(P \textmap{\pi} M,\alpha,A)$ with simply connected base $M$ and geodesically complete induced connection $\nabla^\alpha_A$ can be identified with the triple associated with a real form of a complex Lie group. The main result of \cite{Chou} states that this is the case, see \cite[Theorem 3.3.1]{Chou}. More precisely:

\begin{thr}\label{ThFondam}
Let $M$ be a connected differentiable manifold, $H$ a connected Lie group, and let $\tg = (\pi : P \to M, \alpha, A)$ be an integrable triple. Endow $P$ with the complex structure $J^\alpha_A$ and $M$ with the connection $\nabla_A^\alpha$.

If $\nabla_A^\alpha$ is geodesically complete and $M$ is simply connected, then $\tg$ is associated with  the real form of a complex Lie group. More precisely, there exist:
\begin{enumerate}
\item a connected complex Lie group $G$,
\item a real form $H'$ of $G$ and a Lie group isomorphism $\chi:H' \to H$,
\item a holomorphic principal bundle isomorphism $\Psi: G\to P$ of type $\chi$, 
\item a diffeomorphism $\psi: G/H'\to M$,
\end{enumerate}
such that the diagram
$$\begin{tikzcd} 
    G  \ar[r, "\Psi\simeq "] \ar[d] & P \ar[d, "\pi" ]
    \\  G/  H'
    \ar [r, "\psi \simeq"] & M
\end{tikzcd}
$$
is commutative, and the pair $(\alpha,A)$ corresponds via $(\psi,\Psi)$ to the pair $(\alpha_{H'},A_{H'})$ associated with the real form $H'$ (see section \ref{triples-real-forms}). 
\end{thr}

The proof is long and technical. We explain below the   proof method: \\

The group $G$ (as an abstract group)   will be defined as follows:
$$
G \edf  \{f : P \to P \mid f \text{ is a diffeomorphism}, \ f^*(\alpha) = \alpha, \ f^*(\theta_A) = \theta_A\},
$$
where $\theta_A \in A^1(P, \hg)$ denotes the connection form of $A$. The difficulty is to endow $G$ with a complex Lie group structure that makes the natural action of $G$ on $P$ holomorphic, free, and transitive. We will use  the following theorem  which is of independent interest, and whose proof will be explained in detail below (see also \cite[Theorem 3.3.2, Proposition 3.3.3]{Chou}).
\begin{thr}
	\label{libre&trans} 
Let $M$ be a connected differentiable manifold of dimension $m$, 
$$\rho:G\times M\to M$$
be an effective action by diffeomorphisms of a group $G$ on $M$. 
For $g\in G$, let $\rho_g:M\to M$ denote the diffeomorphism $x\mapsto \rho(g,x)$.

Assume that the set of graphs
$
\{\mathrm{gr}(\rho_g)\mid g\in G\}
$
coincides with the set  $\mathscr{F}$ of leaves of a foliation ${\cal F}$ on $M\times M$. Then:
\begin{enumerate}
	\item the action $\rho$ is free and transitive.
	\item there exists a unique Lie group structure of dimension $m$ on $G$ with respect to which $\rho$ is a differentiable action. 
\end{enumerate}

Suppose that   moreover   $M$ is a complex manifold and that ${\cal F}$ is a holomorphic foliation, i.e., ${\cal F}$ is a holomorphic subbundle of the tangent bundle $T_M$. Then there exists a unique complex Lie group structure on $G$, of the same dimension as $M$, with respect to which the (free and transitive) action $\rho$ is holomorphic.
\end{thr}

We will apply this theorem to the foliation ${\cal F}$ on $P \times P$ associated with the canonical {\it integrable} distribution $\Omega$ on this manifold, defined as follows:

The trivializations $\tau_A^\alpha: A \cong P \times \hg$, $\tau_V: V \cong P \times \hg$ define a global trivialization $T_P \to P \times (\hg \oplus \hg)$ of the tangent bundle $T_P$. In particular, for every $y \in P$, we obtain an isomorphism
$$
f_y: \hg \oplus \hg \to T_yP
$$
so for every pair $(y, z) \in P \times P$, we obtain an isomorphism
$$
f_{yz} \edf f_z \circ f_y^{-1}: T_yP \to T_zP
$$
The distribution $\Omega$ is defined by:
$$
\Omega_{(y,z)} \edf \mathrm{gr}(f_{yz}) \subset T_yP \times T_zP = T_{(y,z)}(P \times P),
$$
Note that a diffeomorphism $f: P \to P$ belongs to $G$ if and only if $\mathrm{gr}(f)$ is an integral submanifold of the distribution $\Omega$.

In order to complete the proof of the classification Theorem \ref{ThFondam} one shows that the set of leaves of $\Omega$ coincides  with the set of graphs $
\{\mathrm{gr}(\rho_g)\mid g\in G\}$. This  is obtained using the following theorem whose proof  is long and technical and goes beyond the goals of this article.
	 
\begin{thr} For every maximal integral submanifold $\omega$ of the distribution $\Omega$, the two natural maps $\omega \to P$ given by the restrictions of the projections $p_i: P \times P \to P$ are diffeomorphisms, in particular, $\omega$ is the graph of a diffeomorphism $f_\omega$ of $P$.
	
\end{thr}

\begin{proof}
See \cite[sections 2.1.2, 4.4 and Proposition 3.3.10]{Chou}.
\end{proof}

Now we  come back to the proof  of the crucial Theorem \ref {libre&trans}:

\begin{proof} (of Theorem \ref {libre&trans}):

(1) To show that $\rho$ is free, let $(g,x)\in G\times M$ such that $\rho(g,x)=x$, i.e., $\rho_g(x)=x$, i.e., $(x,x)\in \mathrm{gr}(\rho_g)$, which by assumption is a leaf of the foliation ${\cal F}$. But the diagonal $\Delta \edf \{(u,u)\mid u\in M\}$ is the graph of $\rho_e = \id_M$, so, by assumption, it is also a leaf of the foliation ${\cal F}$. Hence, $\mathrm{gr}(\rho_g)$ and $\Delta$ are two leaves of ${\cal F}$ passing through the point $(x,x)$. Since the set $\mathscr{F}$ of leaves of a foliation ${\cal F}$ gives a partition of the foliated manifold, it follows that $\mathrm{gr}(\rho_g) = \Delta = \mathrm{gr}(\rho_e)$, thus $\rho_g = \rho_e$, and since the action is effective, we obtain $g = e$.

To show that $\rho$ is transitive, let $(x,y) \in M \times M$. Since 
$\mathscr{F} = \{\mathrm{gr}(\rho_g) \mid g \in G\}$
is a partition of $M \times M$, there exists $g \in G$ such that $(x,y) \in \mathrm{gr}(\rho_g)$, i.e., $y = \rho_g(x)$. 
\vspace{2mm}\\
(2) First observe that ${\cal F}$ is a simple foliation, i.e., the set $\mathscr{F}$ of its leaves admits the structure of a differentiable manifold (not necessarily Hausdorff a priori), such that the natural surjection  
$$q:M\times M\to \mathscr{F}$$
is a submersion. This follows from \cite[Proposition 10.1]{Sh}, noting that for every point $p = (x,y) \in M \times M$, the submanifold $S_x \edf \{x\} \times M$ is a slice (transverse section) to the foliation ${\cal F}$ passing through $p$ and intersects each leaf $\mathrm{gr}(\rho_g)$ at a single point.

Endow $\mathscr{F}$ with the differentiable structure that makes the natural surjection $q: M \times M \to \mathscr{F}$ a submersion. Let $\iota_x: S_x \hookrightarrow M \times M$ denote the inclusion embedding, and observe that for each $x \in X$, the composition
$$f_x:S_x\to \mathscr{F},\ \begin{tikzcd}
S_x\ar[r, hook, "\iota_x"]\ar[rr, bend right, "f_x"] & M\times M\ar[r, two heads ]& \mathscr{F}
\end{tikzcd}
$$
is a diffeomorphism.

\begin{figure}[h]
\centering
\scalebox{0.5}
{\includegraphics{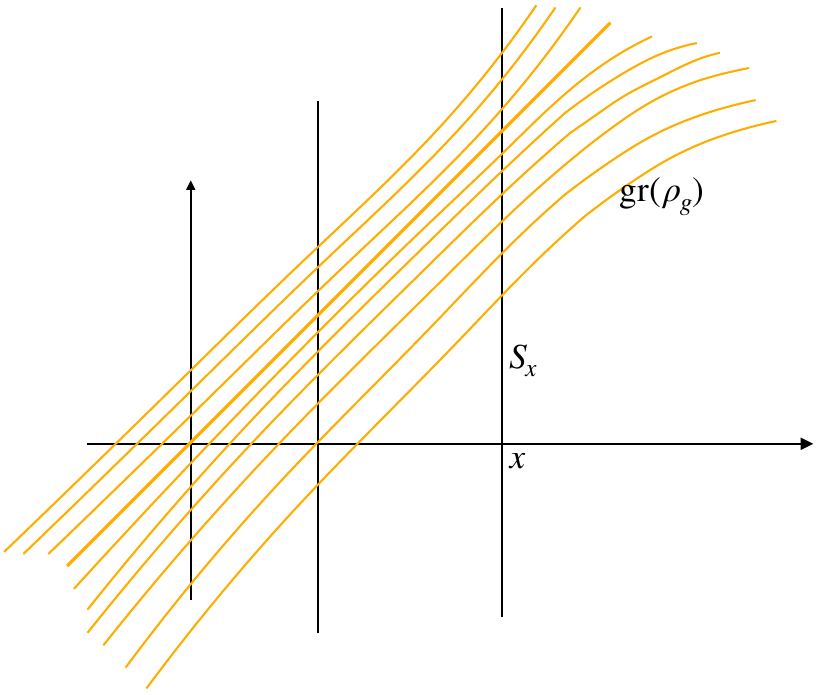}}
\label{branched}
\end{figure}

Indeed, first note that $f_x$ is obviously bijective. Furthermore, since $S_x$ is transverse to the fibers of the submersion $q$, it follows that $f_x$ is a local diffeomorphism. Therefore $f_x$ is a diffeomorphism; in particular, $\mathscr{F}$, being diffeomorphic to $S_x \simeq M$, is a Hausdorff manifold.

The map
$$
\phi:G\to \mathscr{F},\  G\ni g\mapsto  \mathrm{gr}(\rho_g)\in \mathscr{F}
$$
is bijective by hypothesis. Equip $G$ with the differentiable structure that makes this bijection a diffeomorphism. Note that, with respect to this differentiable structure, the map
$$
\rho^x:G\to M, \ \rho^x(g)=\rho(g,x)
$$
is a diffeomorphism for every $x\in M$. Indeed, we have $\rho^x = p_2|_{S_x} \circ f_x^{-1} \circ \phi$, so $\rho^x$ is a composition of three diffeomorphisms.

To conclude, it remains to show that, with respect to this differentiable structure on $G$:

\begin{enumerate}
\item[(a)]   the action $\rho:G\times M\to M$ is differentiable. 
\item[(b)]   the composition map $c:G\times G\to G$ is differentiable.
\item[(c)]   the inversion map $\ig :G\to G$ is differentiable.
\end{enumerate}

For (a), observe that the map $\tilde \rho:  G\times M\to M\times M$, defined by $\tilde \rho(g,x)=(x,\rho(g,x))$, is bijective and, via the identification $\phi:G\to \mathscr{F}$, its inverse is given by
$$
\tilde \rho^{-1}(x,y)=(q(x,y),x).
$$
This map is clearly differentiable. It is also a local diffeomorphism. Indeed, to prove this, we will use the direct sum decompositions
$$T_{(x,y)}(M\times M)=T_{(x,y)}S_x\oplus T_{(x,y)}(\mathrm{gr}(\rho_g)), \quad T_{(q(x,y),x)}(\mathscr{F}\times M)=T_{q(x,y)}\mathscr{F}\oplus T_xM.
$$
With respect to these decompositions, the matrix representation of the tangent map $(\tilde \rho^{-1})_{*(x,y)}$ is
$$
\begin{pmatrix}
	(f_x)_{*(x,y)}   & 0\\ 
	0   &u_{(x,y)}
\end{pmatrix},
$$
where $u_{(x,y)}: T_{(x,y)}(\mathrm{gr}(\rho_g))\to T_x M$ is the tangent map at $(x,y)$ of the restriction of the first projection to the graph $\mathrm{gr}(\rho_g)$. Since this restriction is a diffeomorphism, $u_{(x,y)}$ is an isomorphism. Moreover, $(f_x)_{*(x,y)}$ is an isomorphism because $f_x$ is a diffeomorphism.

Therefore, $\tilde \rho^{-1}$ is a diffeomorphism, so $\tilde \rho$ is a diffeomorphism, and thus its second component $\rho$ is differentiable.
\vspace{2mm}

For (b), fix $x_0\in M$ and note that $\rho^{x_0} \circ c = \rho \circ (\id_M \times \rho^{x_0})$. Since $\rho^{x_0}$ is a diffeomorphism and $\rho$ is differentiable, it follows that $c$ is differentiable.
\vspace{2mm}

For (c), observe that any map of the form
$$
f^y:S^y \edf M\times\{y\}\to \mathscr{F},\ \begin{tikzcd}
S^y\ar[r, hook]\ar[rr, bend right, "f^y"] & M\times M\ar[r, two heads ]& \mathscr{F}
\end{tikzcd}
$$
(where $y\in M$) is also a diffeomorphism (by the same argument as for $f_x$). For fixed $y_0\in M$, we have
$$
p_1|_{S^{y_0}} \circ (f^{y_0})^{-1} \circ \phi = \rho^{y_0} \circ \ig,
$$
thus, since $p_1|_{S^{y_0}}$, $f^{y_0}$, $\phi$, and $\rho^{y_0}$ are diffeomorphisms, $\ig$ will also be a diffeomorphism.
\vspace{4mm} \\
{\it Concerning the holomorphic case:}
\vspace{1mm}

The proof of the second part of Theorem \ref{libre&trans}, given above, easily adapts to the holomorphic case. More precisely, first note that Proposition \cite[Proposition 10.1]{Sh} admits a holomorphic version: let ${\cal F}$ be a holomorphic foliation on a complex manifold $X$; if through every point $x \in X$ passes a {\it complex} submanifold transverse to ${\cal F}$ (in the sense of \cite[Definition 9.2]{Sh}), which intersects each leaf in exactly one point, then the set $\mathscr{F}$ of leaves of ${\cal F}$ admits a structure of complex manifold (a priori not necessarily Hausdorff) such that the projection $X \to \mathscr{F}$ is a holomorphic submersion.

To obtain this result, it suffices to use the definition of holonomy germs \cite[p. 18]{Sh} for foliation charts which are holomorphic, and to note that the resulting holonomy germs will be germs of biholomorphisms.
\end{proof}
\section{A new gauge theoretical moduli problem}

Let $H$ be a Lie group and $\pi:P\to M$ a principal $H$-bundle. We denote by ${\cal A}(P)$ the space of connections on $P$ and we recall (see for instance \cite{Te}) that ${\cal A}(P)$ is an affine space with associated vector space $ A^1_\Ad(P,\hg)=A^1(M,\Ad(P))$, so it can be endowed with a natural topology induced by the Fréchet topology of $A^1(M,\Ad(P))$.

Recall that the gauge group $\Aut(P)$ of $P$ can be identified with the group of section $\Gamma(M,P\times_\iota H)$ of the Lie group bundle $P\times_\iota H$, which is the bundle associated with $P$ and the inner action $\iota:H\times H\to H$, $\iota(\chi,h)\edf \iota_{\chi}h$. The latter space  $\Gamma(M,P\times_\iota H)$ can be further identified with the space of smooth  maps $P\to H$ which are equivariant with respect to the right $H$-action on $P$ and the inner action of $H$ on itself. For a gauge transformation $f\in \Aut(P)$ we will denote by $\tilde f$ the associated equivariant map $P\to H$. Recall also $\Aut(P)=\Gamma(M,P\times_\iota H)$ acts from the left on the  affine space ${\cal A}(P)$  by $f(A)=f_*(A)$. In terms on connection forms, we have the the formula
$$
\theta_{f(A)}=(f^{-1})^*(\theta_A).
$$
$\Aut(P)$ acts also from the left on the spaces $A^r_\Ad(P,\hg)$ of tensorial forms of type $\Ad$ by the formula $f(\alpha)=(f^{-1})^*(\alpha)$. This gives
$$
f(\alpha)(v_1,\dots,v_r)\edf \Ad_{\tilde f(y)}(\alpha (v_1,\dots,v_r)) \hbox{ for }y\in P,\ v_i\in T_yP.
$$
We have the general formulae:
\begin{equation}\label{equiv-prop}
\begin{split}
\Omega_{f(A)}&=f(\Omega_A)  \hbox { for any }  A \in   {\cal A}(P),\\
f(A+\alpha)&=f(A)+f(\alpha) \hbox { for any } (\alpha,A)\in A^1_\ad(P)\times {\cal A}(P),\\
d_{f(A)}(f(\alpha))&=f(d_A\alpha))   \hbox { for any } (\alpha,A)\in A^r_\ad(P)\times {\cal A}(P),\\
f([\alpha\wedge\alpha'])&=[f(\alpha)\wedge f(\alpha')]  \hbox { for any } (\alpha,\alpha')\in A^1_\Ad(P,\hg)\times A^1_\Ad(P,\hg).
\end{split}
\end{equation}

 The two integrability conditions (\ref{ZentEq}) in Zentner's Theorem \ref{ZentTh} have sense for an arbitrary pair 
$(\alpha,A)\in A^1(M,\Ad(P))\times {\cal A}(P)$, without any admissibility condition on $\alpha$. Moreover, using  (\ref{equiv-prop}) it follows easily that the space of solutions 
$$(A^1_\Ad(P,\hg)\times {\cal A}(P))^Z\subset A^1_\Ad(P,\hg)\times {\cal A}(P)$$
of the equations (\ref{ZentEq}) is gauge invariant.

Theorem \ref{ZentTh} states that the solutions of the equations (\ref{ZentEq}) which are admissible have an important geometric interpretation: they are the pairs $(\alpha,A)$ for which the associated almost complex structure $J^\alpha_A$ is integrable.  Therefore it is a natural problem to give a geometric interpretation of the space of gauge equivalence classes of such pairs. Moreover, we believe that these equations and the corresponding moduli space of solutions are interesting without the admissibility condition, so we state the following  general moduli problem:
\begin{prob}
Let $H$ be a Lie group and $\pi:P\to M$ a principal $H$-bundle. Describe the moduli space
$$
{\cal M}^Z(P)\edf \qmod{(A^1_\Ad(P,\hg)\times {\cal A}(P))^Z}{\Aut(P)}
$$	
of solutions of the equations 
\begin{equation} \label{ZentEqNew}
\left\{\begin{array}{ccc}
d_A \alpha &  = & 0 \\ 
\Omega_A & =&   \frac{1}{2}[ \alpha \wedge \alpha].
\end{array}\right. \tag{$Z$}
\end{equation}
\end{prob}

The moduli space ${\cal M}^Z(P)$  will be regarded as a topological space obtained by endowing ${(A^1_\Ad(P,\hg)\times {\cal A}(P))^Z}/{\Aut(P)}$ with the quotient topology,  the space of solutions $(A^1_\Ad(P,\hg)\times {\cal A}(P))^Z$ being a subspace of the product $A^1_\Ad(P,\hg)\times {\cal A}(P)$ which has a natural Fréchet topology.

Note that our moduli problem (\ref{ZentEqNew}) has obvious similarities with classical moduli problems, e.g. with the moduli problem introduced and studied by Hitchin in the renowned article  \cite{Hi}.  Hitchin's  self-duality equations are equations for a pair consisting  of a connection on a principal bundle $P$ with compact Lie group on a Riemann surface and ``a Higgs field'' of type $P$, i.e. an $\Ad(P)$-valued (1,0)-form.   In our moduli problem, exactly as in Hitchin's moduli problem: \begin{itemize} 
 \item	we have two unknowns, namely a  1-form $\alpha$ and a connection $A$,
 \item  the first equation is a first order equation for the form $\alpha$,
 \item the second equation expresses the curvature of $A$  as a 0-order quadratic expression in $\alpha$.
 \end{itemize}
 Note however that the deformation complex of our moduli problem at a solution $(\alpha,A)$ is not elliptic, so one cannot expect a  finite  dimensional moduli  space.
\vspace{1mm}

We also define the open subspace  
$$ {\cal M}^Z_a(P)\edf \big \{[\alpha,A]\in  {\cal M}^Z (P)|\ \alpha\hbox{ is admissible}\big\}\subset{\cal M}^Z (P),  $$
of integrable admissible pairs; this space is of course empty if $\dim(M)\ne \dim(H)$.

${\cal M}^Z(P)$ (${\cal M}^Z_a(P)$) will be called the moduli space of (admissible) integrable pairs on the principal $H$-bundle $P$.

The formalism developed in \cite[chapter 2]{Chou} will allows one to solve in part our moduli problem, more precisely  to give an explicit description of  ${\cal M}^Z_a(P)$ in the special case when $H$ is a compact semisimple compact Lie group (see  \cite{ChouTe}).  We illustrate below our methods  in the special cases  $H=\SO(3)$, $H=\SU(2)$. 

\subsection{Examples of moduli spaces}

Recall that, by the renowned Stiefel's theorem, any compact  orientable 3-manifold is parallelisable \cite[Problem 12-B, p. 148]{MiSt}.  The Lie algebra $\so(3)$ comes with a canonical orientation and a canonical inner product which are induced by the isomorphism 
$$\R^3\ni u\to (u\times \cdot)\in  \so(3).$$
Similarly, the adjoint bundle $\Ad(P)=P\times_\Ad\so(3)$ of a principal $\SO(3)$-bundle $P\textmap{\pi} M$ comes with a canonical orientation and a canonical  Euclidian structure which is $\nabla_A^\Ad$-parallel for any connection $A$ on $P$.
\vspace{2mm} 

Note that:

\begin{re}\label{RemSO3}

Let $P\textmap{\pi} M$ be principal $\SO(3)$-bundle. For a gauge transformation $f\in \Aut(P)$ denote by $\Ad(f)$ the induced vector bundle automorphism of $\Ad(P)$. 
\begin{enumerate}
\item 	The map $f\mapsto \Ad(f)$ is an isomorphism between $\Aut(P)$  and the group of orientation preserving isometric automorphisms of $\Ad(P)$.	
\item The map $A\mapsto \nabla_A^\Ad$ is an equivariant affine isomorphism from ${\cal A}(P)$ onto the affine space of linear connections on $\Ad(P)$ which are compatible with its canonical Euclidian structure.
\end{enumerate}
\end{re}
We can now prove:
\begin{thr}\label{ModuliSO(3)}
Let $M$ be a closed orientable 3-manifold and $\pi:P\to M$ a principal $\SO(3)$-bundle on $M$.  Then
\begin{enumerate}
\item  If $P$ is non-trivial, then 	${\cal M}^Z_a(P)$ is empty.	
\item  If $P$ is trivial, then the moduli space ${\cal M}^Z_a(P)$ can be naturally identified with the product ${\cal O}(M)\times{\cal H}(M)$, where ${\cal O}(M)$ is the set of orientations of $M$ and ${\cal H}(M)$ is the space of hyperbolic Riemannian metrics on  $M$.
\end{enumerate}

\end{thr}

\begin{proof}
(1) By Remark \ref{alpha-beta} we know that an admissible form $\alpha\in A^1_\Ad(P,\so(3))$ induces a bundle isomorphism $\beta:T_M\to \Ad(P)$. Since $T_M$ is trivial by Stiefel's theorem, it follows that $\Ad(P)$ is a trivial  vector bundle of rank 3. But the classification of oriented real vector bundles coincides with the classification of oriented Euclidian vector bundles, so  $\Ad(P)$ is trivial as oriented Euclidian bundle.

On the other hand, the functor $P\mapsto \Ad(P)$ is an equivalence between the groupoid of principal $\SO(3)$-bundles and the  groupoid of  oriented Euclidian vector bundles. Therefore, if an admissible form $\alpha$ exists, then $P$ must be trivial. \vspace{2mm}

(2) Suppose now that $P$ is trivial. Let $(\alpha,A)$ be a solution of the equations (\ref{ZentEqNew}) with $\alpha$ admissible. Via the isomorphism $\beta:T_M\to \Ad(P)$ associated with $\alpha$,   the canonical oriented Euclidian bundle structure on $\Ad(P)$ defines an oriented Euclidian bundle structure on $T_M$, so an orientation $o_\alpha$ and a Riemannian metric $g_\alpha$ on $M$. Since the canonical Euclidian structure on $\Ad(P)$ is $\nabla_A^\Ad$-parallel, it follows that $g_\alpha$ is $\nabla_A^\alpha$-parallel, so, since $T^{\nabla_A^\alpha}=0$ by Proposition \ref{courbure-nablaM}, $\nabla_A^\alpha$ is the Levi-Civita connection of $g_\alpha$. The curvature formula given by the same proposition now shows that $g_\alpha$ is hyperbolic.
\vspace{1mm}

Now note that the pair $(o_\alpha,g_\alpha)$ depends only on the gauge equivalence class of $(\alpha,A)$. Indeed, let $f\in \Aut(P)$ and $\alpha'=f(\alpha)$. The isomorphism $\beta':T_M\to \Ad(P)$ associated with $\alpha'$ is given by $\beta'=\Ad(f)\circ \beta$. The claim follows, because $\Ad(f)$ is orientation preserving and isometric.

Therefore we obtained a well defined map
$$
\Psi: {\cal M}^Z_a(P)\to {\cal O}(M)\times{\cal H}(M), \ \Psi([\alpha,A])\edf (o_\alpha,g_\alpha).
$$

{\it  $\Psi$ is injective:}

 Indeed, let $(\alpha,A)$, $(\alpha',A')\in (A^1_\Ad(P,\hg)\times {\cal A}(P))^Z$ such that $(o_{\alpha'},g_{\alpha'})=(o_\alpha,g_\alpha)$. It follows that the vector bundle automorphism $\beta'\circ \beta^{-1}$  is orientation preserving and preserves the canonical  Euclidian structure of $\Ad(P)$, so, by Remark \ref{RemSO3}, it is induced by an automorphism $f\in\Aut(P)$. We have $\beta'=\Ad(f)\circ \beta$, which implies $\alpha'=f(\alpha)$.  It remains to prove that $A'=f(A)$. Let $\nabla$ be Levi-Civita connection of $g_\alpha=g_{\alpha'}$. We have
$$
\nabla_{A'}^\Ad=\beta'(\nabla)=\Ad(f)(\beta(\nabla))=\Ad(f)(\nabla_A^\Ad),
$$
so $\nabla_A^{\alpha'}=\Ad(f)(\nabla_A^{\alpha})$.  By Remark \ref{RemSO3},   the equality $\nabla_{A'}^\Ad=\Ad(f)(\nabla_A^\Ad)$ implies $A'=f(A)$, and the injectivity of $\Psi$ is proved. 
\vspace{2mm}\\
{\it  $\Psi$ is surjective:}

Indeed, let $(o,g)\in {\cal O}(M)\times {\cal H}(M)$. Therefore $(M,o,g)$ becomes an oriented hyperbolic 3-manifold. By Corollary \ref{hyperbolic-coro}, the principal bundle $P_g^o$ of $o$-compatible orthonormal frames of the tangent bundle comes with a canonical admissible integrable pair $(\alpha_0,A_0)$, where $A_0$ is the Levi-Civita connection of $g$ and $\alpha_0$ is given by $\alpha_0(u)(v)=u\times v$. By Stiefel's theorem we know that $P_g^o$ is trivial. Since $P$ is also trivial by assumption, there exists a bundle isomorphism $h:P\to P_g^o$. Putting $\alpha\edf h^*(\alpha_0)$, $A\edf h^*(A_0)$, we obtain an integrable admissible pair $(\alpha,A)$ on $P$, and obviously $\Psi([\alpha,A])=(o,g)$.

\end{proof}

Let $\theta:\Spin(n)\to \SO(n)$ be the canonical Lie group epimorphism and let $(M,o,g)$ be an {\it oriented}  Riemannian $n$-dimensional manifold.   We refer to \cite{LM}, \cite{Te} for the following well known definitions and results:
\begin{dt}
A 	$\Spin$ structure on $(M,o,g)$ is a principal bundle morphism $P\textmap{\gamma} P_g^o$ of type $\theta$,  where $P$ is principal $\Spin(n)$-bundle and $P_g^o$ is the principal $\SO(n)$-bundle of $o$-compatible orthonormal frames. Two $\Spin$ structures  $P\textmap{\gamma} P_g^o$,  $P'\textmap{\gamma'} P_g^o$ are called equivalent if there exists an isomorphism  $f:P\to P'$ of $\Spin(n)$-bundles such that $\gamma'\circ f=\gamma$.  
\end{dt}

\begin{re}
An oriented  Riemannian manifold $(M,o,g)$ admits a $\Spin$ structure if and only if $w_2(T_M)=0$. If this is the case, the set $\Spin(M,o,g)$ of equivalence  classes of $\Spin$ structures on $(M,o,g)$ is naturally an $H^1(M,\Z_2)$-torsor. 	
\end{re}

\begin{re}
Fix an orientation $o$ on $M$. Suppose that $w_2(T_M)=0$. 

For two Riemannian metrics $g$, $g'$ on $M$ we have a canonical isomorphism $\Spin(M,o,g)\textmap{\simeq}\Spin(M,o,g')$ of $H^1(M,\Z_2)$-torsors, so one can define in a coherent way the torsor $\Spin(M,o)$ of equivalence classes of $\Spin$ structures on the oriented manifold $(M,o)$.
\end{re}

For an {\it orientable} manifold $M$ we define
\begin{equation} \label{Spin(M)}
\Spin(M)\edf\{(o,\sigma)|\ o\in {\cal O}(M),\ \sigma\in \Spin(M,o)\}.	
\end{equation}

In dimension 3 we have the following important properties:

\begin{re}
The adjoint morphism $\Ad:\SU(2)\to \SO(\su(2))$ is a double cover, so it identifies $\SU(2)$ with $\Spin(3)$.
\end{re}

\begin{pr}
Any topological principal $\SU(2)$-bundle on a  CW complex of dimension $d\leq 3$ is trivial. 
\end{pr}
\begin{proof}
Let $X$ be a CW complex of dimension $d\leq 3$. It suffices to show that any  principal $\SU(2)$-bundle on $X$ admits a continuous section. This follows using Steenrod's obstruction theory \cite{St} taking into account that $\pi_k(\SU(2))=0$ for $1\leq k\leq 2$.	
\end{proof}

Taking into account that on a differentiable manifold the classification of differentiable bundles coincides with  the classification of topological bundles, we infer:
\begin{co}
Any differentiable $\SU(2)$-bundle on a differentiable manifold of dimension $d\leq 3$ is trivial.	
\end{co}

Using these definitions and results, the method used in the proof of Theorem \ref{ModuliSO(3)} gives:

\begin{thr}\label{ModuliSU(2)}
Let $M$ be a closed orientable 3-manifold and $\pi:P\to M$ a principal $\SU(2)$-bundle on $M$.  Then the moduli space ${\cal M}^Z_a(P)$ can be naturally identified with the product $\Spin(M)\times {\cal H}(M)$, where    ${\cal H}(M)$ is the space of hyperbolic Riemannian metrics on $M$. 

\end{thr}

We conclude this article  by stating the following problem which concerns moduli spaces of  non-admissible integrable pairs.
\vspace{1mm}

 Let  $P\textmap{\pi}M$ be a principal $H$-bundle. A tensorial 1-form $\alpha\in A^1_\Ad(P,\hg)=A^1(M,\Ad(P))$ can be regarded as a section in the bundle $\Hom(T_M,\Ad(P))$. Suppose  $\dim(M)=\dim(H)=n$. In this case we can define  
$$ \det(\alpha)\in \Gamma\big(M, \Hom(\extp^n(T_M),\extp^n(\Ad(P))\big)=\Gamma\big(M,\extp^n(T_M)^\vee\otimes \extp^n(\Ad(P)\big).$$
Note that $\alpha$ is admissible if and only if $\det(\alpha)$ is a nowhere vanishing section of the real line bundle $\extp^n(T_M)^\vee\otimes \extp^n(\Ad(P)$. If $\det(\alpha)$ is transversal to the zero section (i.e. its intrinsic derivative is surjective at any vanishing point), then its zero locus $Z(\det(\alpha))$ will be a smooth hypersurface of $M$.

\begin{prob}
Let $M$ be a closed 3-manifold, $S\subset M$ a fixed closed submanifold  of dimension 2, and $P\textmap{\pi}M$ a principal $\SO(3)$-bundle on $M$. Describe explicitly the moduli space of integrable pairs $(\alpha,A)\in {(A^1_\Ad(P,\hg)\times {\cal A}(P))^Z}$, where $\det(\alpha)$ is transversal to the zero section and $Z(\det(\alpha))=S$.  
	
\end{prob} 

Note that such an integrable pair defines a hyperbolic metric $g_\alpha$ on the complement $M\setminus S$. A first step in solving this problem is to study the asymptotics of $g_\alpha$ near $S$.

\end{document}